\documentclass[10pt]{amsart}
\usepackage{amsfonts,amsmath,amsthm, amscd, amssymb, cite, color, latexsym, mathrsfs, mathtools,
slashed, stmaryrd, verbatim, wasysym}
\usepackage[colorlinks=true,linkcolor=blue,urlcolor=blue,citecolor=blue]{hyperref}
\usepackage[all]{xy}

\renewcommand{\tilde}{\widetilde}

\newcommand\eps\varepsilon
\renewcommand\epsilon\varepsilon

\newcommand\Ch{\operatorname{Ch}}

\newcommand\reg{\operatorname{reg}}
\newcommand\sing{\operatorname{sing}}

\newcommand\paperintro%
        {%
         }
\newcommand\paperbody%
        {%
         }

\DeclareMathAlphabet{\mathpzc}{OT1}{pzc}{m}{it}

\begin{document}

\newtheorem{theorem}{Theorem}[section]
\newtheorem{lemma}[theorem]{Lemma}
\newtheorem{proposition}[theorem]{Proposition}
\newtheorem{corollary}[theorem]{Corollary}

\theoremstyle{definition}
\newtheorem{definition}[theorem]{Definition}
\newtheorem{remark}[theorem]{Remark}
\newtheorem{example}[theorem]{Example}
\newtheorem{problem}[theorem]{Problem}
\newtheorem*{acknowledgements}{Acknowledgements}

\numberwithin{equation}{section}

\renewcommand{\subjclassname}{%
\textup{2010} Mathematics Subject Classification}

\title[A note on higher Todd genera of complex manifolds]{A note on higher Todd genera of complex manifolds}

\author[F. Bei and P. Piazza]{Francesco Bei and Paolo Piazza}


\address{Dipartimento di Matematica, Sapienza Universit\`a di Roma}
\email{bei@mat.uniroma1.it}

\address{Dipartimento di Matematica, Sapienza Universit\`a di Roma}
\email{piazza@mat.uniroma1.it}

\maketitle

\begin{abstract}
Let $M$ be a compact complex manifold. In this paper we give a simple proof of the  bimeromorphic invariance of the higher Todd genera of $M$, a result first proved implicitly in \cite{BraSchYou} using algebraic methods.
\end{abstract}

\tableofcontents

\section{Introduction}

 Let $M$ be a smooth complex projective variety. It is well known that  $\pi_1 (M)$ is a birational invariant of $M$. Let ${\rm Td}(M)\in H^* (M,\mathbb{Q})$ be the Todd class of $M$. Let $s:M\to B\pi_1 (M)$ be a classifying map for the universal cover of $M$. The higher Todd genera of $M$ are defined as the rational numbers
 $$\{({\rm Td}(M)\cup s^* \alpha)[M]\,,\;\;\alpha\in H^* (B\pi_1 (M),\mathbb{Q})\}.$$
Jonathan Rosenberg \cite{rosenberg-tams}, building on a well-established argument for  proving  the oriented homotopy invariance of the higher signatures, proved  that if the assembly map $\beta: K_0 (B\pi_1 (M))\to K_0 (C^*_r \pi_1 (M))$ is rationally injective then the higher Todd genera are birational invariants. 
This result had been in fact proved unconditionally, i.e. without assuming the rational injectivity of $\beta$,
 in the preprint of Brasselet-Yokura-Sch\"urmann \cite{BSYarxiv},
later published in \cite{BraSchYou}.
Later, Block and Weinberger also proved the latter  result, see \cite{WBlock}. 
Yet another proof was given by Hilsum, using analytic methods \cite{Hilsumproj}.
All these articles use in a crucial way the weak factorization theorem for birational maps \cite{Factorization}.
 
In this short note we have two goals in mind. First we 
give a simple proof that the higher Todd genera are bimeromorphic invariants for smooth compact  complex manifolds,
a result already proved implicitly in \cite{BraSchYou} using a {\it motivic Chern class transformation}. Our proof  
 does not use the weak factorization theorem but relies instead on the notion of modification; moreover 
 our definition of bimeromorphic invariance for higher Todd genera is more general than the one adopted in all 
 of these articles. See Section \ref{sect:bimer}.
 Secondly, we extend these results to suitably defined (analytic) higher Todd genera for certain 
 complex analytic spaces with isolated singularities.

\medskip
\noindent
{\bf Acknowledgements.} 
This paper was partially written while the first author was a postdoc at the department of Mathematics of the University of Padova. He wishes to thank that institution  for financial support. We thank Jonathan Rosenberg and J\"org Sch\"urmann for useful comments and interesting discussions.

\section{Meromorphic maps and fundamental groups}

We recall some definitions and properties that will play a central role in the paper. Since it will be frequently used later on we remind  that an irreducible complex space $X$ is a reduced complex space such that $\reg(X)$, the regular part of $X$, is connected. We have now the following

\begin{definition}
\label{modification}
Let $W$ and $V$ be two  irreducible complex spaces. A proper and surjective holomorphic map $f:W\rightarrow V$ is a proper modification if there exists a nowhere dense analytic subset $X\subset V$ such that $Y=f^{-1}(X)$ is a nowhere dense analytic subset of $W$ and $f|_{W\setminus Y}:W\setminus Y\rightarrow V\setminus X$ is a biholomorphism.
\end{definition}
If $W$ and $V$ are compact then we will simply say that $f:W\rightarrow V$ is a modification.

\begin{definition}
\label{mero}
Let $W$ and $V$ be two irreducible complex spaces. A meromorphic map $f:W\dashrightarrow V$ is a map from $W$ to $P(V)$, the power set of $V$, such that 
\begin{enumerate}
\item $\overline{\mathcal{G}(f)}$, defined as the closure in $W\times V$ of $\mathcal{G}(f)=\{(x,y)\in W\times V\ such\ that\ y\in f(x)\}$, is an irreducible analytic subvariety of $W\times V$,
\item The natural projection $p_W:\overline{\mathcal{G}(f)}\rightarrow W$ is a proper modification.
\end{enumerate}
The map is called bimeromorphic if also $p_V:\overline{\mathcal{G}(f)}\rightarrow V$, the natural projection on $V$, is a modification.
\end{definition}
For more details see e.g   \cite{Peter}, \cite{Tomassini} or \cite{Ueno}. Definition \ref{mero} implies the existence of a smallest analytic subset $Z\subset W$, usually called the set of points of indeterminacy of $f$, such that $f$ is defined and holomorphic on $W\setminus Z$. We shall denote the open set $W\setminus Z$ as ${\rm Dom} (f)$.
If both $W$ and $V$ are normal, a  fundamental property is that the set of points of indeterminacy of $f$ has complex codimension at least $2$, see \cite{Ueno} Th. 2.5. Clearly the composition of two modifications is still a modification and any modification is a bimeromorphic map, see \cite[Rmk 1.8]{Peter}. In particular if $f:M\rightarrow N$ is proper modification of complex manifolds then the set of points of indeterminacy of $f^{-1}:N\dashrightarrow M$ has complex codimension at least 2. The following result is well known to the experts but as we could not find a quotable reference, we provide a proof for the benefit of the reader.

\begin{proposition}
\label{isofgroup}
Let $f:M\rightarrow N$ be a  modification between two compact  complex manifolds. Then $f_*:\pi_1(M)\rightarrow \pi_1(N)$ is an isomorphism
\end{proposition}

\begin{proof}
Let $X\subset N$ and $Y\subset M$ be the smallest analytic subsets such that $f|_{M\setminus Y}:M\setminus Y\rightarrow N\setminus X$ is a biholomorphism. Thanks to \cite{Chirka} page 60 we can decompose $X$ as $X=\bigcup_{k=1}^{\ell} S_k$ such that 
\begin{itemize}
\item $S_k$ is a complex submanifold of $N$ for each $k=1,...,\ell$ and $S_k\cap S_j=\emptyset$ whenever $j\neq k$,
\item For each $k=1,...,\ell$ both $\overline{S}_k$ and $\overline{S}_k\setminus S_k$ are analytic subsets of $N$,
\item If $S_j\cap \overline{S}_k\neq \emptyset$ and $S_j\neq S_k$ then $S_j\subset \overline{S}_k$ and $\dim(S_j)<\dim(S_k)$.
\end{itemize}
Without loss of generality we can assume that $S_1,...,S_{\ell}$ are ordered in such a way that $\dim(S_i)\leq \dim(S_j)$ if $i\leq j$. It is easy to verify that the above properties imply that $S_1\cup...\cup S_{k}$ is closed in $N$ for each $k=1,...,\ell$. In particular  $N\setminus (S_1\cup...\cup S_{k})$ is  a complex manifold, in fact it is an open subset of $N$,  and $S_k$ is a closed complex submanifold of $N\setminus (S_1\cup...\cup S_{k-1})$. Moreover, as remarked above, we also know that the complex  codimension of $S_k$ satisfies $\mathrm{codim}_{\mathbb{C}}(S_k)\geq 2$ for each $k=1,...,\ell$. This follows by the fact that $X$ is the set of points of indeterminacy of $f^{-1}:N\dashrightarrow M$.
Clearly $Y$ has an analogous stratification in $M$ whose strata will be denoted with $T_1,...,T_r$. Also in this case we will assume that $\dim(T_i)\leq \dim(T_j)$ if $i\leq j$ and as in the previous case we have that  $T_1\cup...\cup T_{k}$ is closed in $M$ for each $k=1,...,r$.  As $\mathrm{codim}_{\mathbb{C}}(S_{k})\geq 2$ for any $k$ a well known application of Thom's transversality theorem tells us that the inclusion $(N\setminus S_{1})\hookrightarrow N$ induces an isomorphism $\pi_1(N\setminus S_{1})\cong \pi_1(N)$. Consider now $S_2$. Since it is a closed submanifold of $N\setminus S_1$ we have that $N\setminus (S_1\cup S_2)$ is still a (complex) manifold, and thus Thom's transversality theorem tells us that the inclusion $N\setminus (S_2\cup S_{1})\hookrightarrow N\setminus S_1$ induces an isomorphism $\pi_1(N\setminus (S_2\cup S_{1}))\cong \pi_1(N\setminus S_1)$. If we iterate this procedure at the $k$-th step we have $S_k$, which is a closed complex submanifold of   $N\setminus (S_1\cup...\cup S_{k-1})$, and again Thom's transversality theorem tells us that the inclusion $N\setminus (S_1\cup...\cup S_{k})\hookrightarrow N\setminus (S_1\cup...\cup S_{k-1})$ induces an isomorphism $\pi_1(N\setminus (S_{1}\cup...\cup S_{k}))\cong \pi_1(N\setminus (S_{1}\cup...\cup S_{k-1}))$. Finally after $\ell$-times we obtain that the inclusion $N\setminus (S_1\cup...\cup S_{\ell})\hookrightarrow N\setminus (S_1\cup...\cup S_{\ell-1})$ induces an isomorphism $\pi_1(N\setminus (S_{1}\cup...\cup S_{\ell}))\cong \pi_1(N\setminus (S_{1}\cup...\cup S_{\ell-1}))$. Composing all these maps and the corresponding isomorphisms we get that the inclusion $N\setminus X\hookrightarrow N$ induces an isomorphism $\pi_1(N\setminus X)\cong \pi_1(N)$. Moreover the same strategy applied to $M$ and $Y$ tells us that the inclusion $M\setminus Y\hookrightarrow M$ induces a surjective morphism $\pi_1(M\setminus Y)\rightarrow \pi_1(M)$. We remark that in this case we get a different result (in fact weaker as $\pi_1(M\setminus Y)\rightarrow \pi_1(M)$ is only an epimorphism) because, concerning the codimension of $Y$,  we only know that $\mathrm{codim}_{\mathbb{C}}(Y)\geq 1$. Therefore, at each step, Thom's transversality theorem tells us only that the inclusion $M\setminus (T_1\cup...\cup T_{k})\hookrightarrow M\setminus (T_1\cup...\cup T_{k-1})$ induces a surjective morphism $\pi_1(M\setminus (T_{1}\cup...\cup T_{k}))\cong \pi_1(M\setminus (T_{1}\cup...\cup T_{k-1}))$. 
Finally let us now denote by $i$ and $j$ the inclusions $j:N\setminus X\hookrightarrow N$ and  $i:M\setminus Y\hookrightarrow M$, respectively. We know that $(f\circ i)_*=(j\circ (f|_{M\setminus Y}))_*$. As  $(j\circ (f|_{M\setminus Y}))_*:\pi_1(M\setminus Y)\rightarrow \pi_1(N)$ is an isomorphism and $i_*:\pi_1(M\setminus Y)\rightarrow \pi_1(M)$ is surjective we can conclude that $f_*:\pi_1(M)\rightarrow \pi_1(N)$ is an isomorphism as desired.
\end{proof}

\begin{corollary}\label{cor:iso-bimero}
Let $\phi:M\rightarrow N$ be a bimeromorphic map between two  compact complex  manifolds. Then $\phi$ induces an isomorphism $\phi_*:\pi_1(M)\rightarrow \pi_1(N)$.
\end{corollary}
\begin{proof}
We use the notations of Def. \ref{mero}. Let $\pi:L\rightarrow \mathcal{G}(f)$ be a resolution of $\mathcal{G}(f)$. Then $p_M\circ \pi:L\rightarrow M$ and $p_N\circ \pi:L\rightarrow N$ are both modifications. Now the statement is an immediate consequence of Prop. \ref{isofgroup}.
\end{proof}

\section{The Levy-Riemann-Roch Theorem}
We begin by recalling some fundamental facts about modifications.

\begin{theorem}\label{theorem-ueno}
Let $p: L\rightarrow M$ be a proper modification of complex manifolds. Then\\ 
(i) $p_* \mathcal{O}_L=\mathcal{O}_M$;\\ (ii) $R^k p_* \mathcal{O}_L=0$ for $k>0$\,.
\end{theorem}

\begin{proof}
See \cite{Ueno} Cor. 1.14 and Prop 2.14, respectively.
\end{proof}

We also recall a particular version of Levy's  Riemann-Roch theorem \cite{Roni}. Let $M$ be a  complex manifold. Consider $K_0^{{\rm hol}} (M)$, the  Grothendieck group of coherent analytic sheaves on $M$. Let  $K^{\text{top}}_{0}(M)$ be the topological K-homology of $M$. Then there exists a homomorphism of abelian groups $\alpha_M: K^{\text{hol}}_0 (M)\rightarrow K^{\text{top}}_{0}(M)$   such that, in particular, the following holds:\\
if $f:M\to N$ is a proper holomorphic map between complex manifolds and $f_{!}:  K^{\text{hol}}_0 (M)\to  K^{\text{hol}}_0 (N)$ is  the direct image homomorphism provided by Grauert's theorem, then 
\begin{equation}\label{direct}
f_* (\alpha_M [\mathcal{O}_M ])= \alpha_N (f_{!} [\mathcal{O}_M ])\,.
\end{equation}
Thus if  $p: L\to M$ is a proper modification of complex manifolds we have 
$$ p_* (\alpha_L [\mathcal{O}_L])=\alpha_M [p_{!} \mathcal{O}_L]=\alpha_M [p_{*} \mathcal{O}_L]=\alpha_M [\mathcal{O}_M]\quad \text{in}\quad K^{\text{top}}_{0}(M)$$ where the first equality comes from \eqref{direct}, the second from (ii) of Theorem \ref{theorem-ueno}  and the third from (i) of Theorem \ref{theorem-ueno}.\\

Moreover we recall that if $M$  is a compact  complex manifold then  the image of  $\alpha_M [\mathcal{O}_M ]\in K^{\text{top}}_{0}(M)$ under the isomorphism  $$K^{\text{top}}_{0}(M)\to KK_0 (C(M),\mathbb{C})$$
is precisely $[\overline{\eth}_M]$, the analytic K-homology class associated to the operator $\overline{\partial}_0+ \overline{\partial}_0^t$  acting on $\Omega^{0,\bullet}(M):=\bigoplus_{q=0}^m \Omega^{0,q}(M)$. See \cite{BeiPiazza2019}, p. 35.

\medskip
\noindent
{\bf Notation}: we set $ K_0^{\text{an}} (M):= KK_0 (C(M),\mathbb{C})$.

\medskip
\noindent

Summarizing, if 
 $p: L\rightarrow M$ is  a  proper modification of  complex manifolds, then 
 \begin{equation}\label{equality}
 p_* (\alpha_L [\mathcal{O}_L])=
 \alpha_M [\mathcal{O}_M]\;\;\text{in}\;\;K^{\text{top}}_{0}(M) 
\end{equation}
and if in addition we require both $L$ and $M$ to be compact then
  \begin{equation}\label{equality-bis}
 p_* [\overline{\eth}_L]=[\overline{\eth}_M]\;\;\text{in}\;\; K_0^{\text{an}} (M)\,.
 \end{equation}
 These equalities will be crucial in what follows.

\section{Bimeromorphic invariance}\label{sect:bimer}

We begin this section by explaining what we mean by bimeromorphy invariance of the higher Todd  genera.
To this end we first recall the Novikov conjecture on the oriented homotopy invariance of the higher signatures.
Let $N$ and $M$ be oriented smooth compact manifolds. If $\Gamma$ is a finitely generated discrete group
and $r:N\to B\Gamma$ is a continuous map, 
then the higher signatures of $N$ are the collection of numbers
\begin{equation}\label{higher-signatures}
\left\{\int_N {\rm L}(N)\wedge r^* \alpha\,,\quad \alpha\in H^*(B\Gamma,\mathbb{Q})\right\}\end{equation} where $L(N)$ denotes the Hirzebruch $L$-class of $N$. By homotopy invariance of these numbers we mean the following: given an orientation preserving 
homotopy equivalence $M\xrightarrow{f}N$, the following equality 
$$\int_N {\rm L}(N)\wedge r^* \alpha=\int_M  {\rm L}(M)\wedge (r\circ f)^* \alpha$$
holds for any $\alpha\in  H^*(B\Gamma,\mathbb{Q})$. The Novikov conjecture asserts that all the numbers
\eqref{higher-signatures} are oriented homotopy invariants. It has been proved for large classes of 
finitely generated discrete groups but remains open in its generality.

Consider now two  compact complex manifolds $M$ and $N$ and a bimeromorphism $f:M\dasharrow N$.  
When we try to follow the above formulation in order to define the bimeromorphic invariance 
of the higher Todd genera we face the problem that $f$, in contrast with the Novikov case,
 is {\it not} everywhere defined. 
We could define the bimeromorphy invariance of the Todd higher genera as follows:
 if $s:M\to B\Gamma$ and $r: N\to B\Gamma$ are two continuous maps such that $s=r\circ f$
 on the dense open subset of $M$ where $f$ is defined, then
 $$\int_N {\rm Td}(N)\wedge r^* \alpha=\int_M   {\rm Td}(M)\wedge s^* \alpha$$
holds for any $\alpha\in  H^*(B\Gamma,\mathbb{Q})$.
This  is how birational invariance for the higher Todd genera is formulated for example  in \cite{Hilsumproj} and, implicitly, in  \cite{BraSchYou}  \cite{rosenberg-tams} \cite{WBlock}, in the context of smooth projective varieties. Recall that this invariance, which holds without additional hypothesis on $\Gamma$, is proved in these papers using
in a fundamantel way  the weak factorization theorem \cite{Factorization}.

In this article we follow a  more general  formulation, based on Corollary \ref{cor:iso-bimero}.
So our goal in the first part of this section is to reformulate in a more general way the bimeromorphic invariance of the higher Todd genera and to   establish it for smooth compact complex manifolds. 

\begin{definition}\label{main-definition}
Let $s:M\to B\Gamma$  be any continuous  map. By bimeromorphic invariance of the higher Todd genera associated to $M$
and $s:M\to B\Gamma$,
$$\left\{\int_M {\rm Td}(M)\wedge s^* [c]\,,\quad [c]\in H^*(B\Gamma,\mathbb{Q})\right\}\,,$$
we mean the equality
\begin{equation}
\label{2wemean}
\int_M {\rm Td}(M)\wedge s^* [c]= \int_N {\rm Td}(N)\wedge r^* [c]
\end{equation}
for each $[c]\in   H^*(B\Gamma,\mathbb{Q})$ and for each continuous map $r:N\rightarrow B\Gamma$ 
satisfying \begin{enumerate}
\item $s(p)=r(f(p))$ for some $p\in \mathrm{Dom}(f)$,
\item $r\circ f:\mathrm{Dom}(f)\rightarrow B\Gamma$ is homotopic to $s|_{\mathrm{Dom}(f)}:\mathrm{Dom}(f)\rightarrow B\Gamma$  through a homotopy fixing $p$.
\end{enumerate}
\end{definition}
It is clear that this definition of bimeromorphic invariance is more general than the one in \cite{Hilsumproj}
\cite{rosenberg-tams} \cite{WBlock} \cite{BraSchYou} , in that
it allows a larger set of compatible maps into $B\Gamma$. Consequently, the birational invariance
or more generally the bimeromorphic invariance of the higher Todd genera proved in this article
is a stronger invariance-property compared to the one established in 
\cite{Hilsumproj}
\cite{rosenberg-tams} \cite{WBlock} \cite{BraSchYou}.

 \smallskip
 \noindent
 The following two results are the crucial steps in establishing the bimeromorphic invariance of the higher Todd genera.

 \begin{proposition}\label{homotopy}
Let $f:M\dasharrow N$ be a bimeromorphic map between two compact complex manifolds and let $Z$ be any $K(\Gamma,1)$ space, $\Gamma$ any discrete finitely generated group. Let $s:M\rightarrow Z$ be any continuous map. We have the following properties:\\
 For any arbitrarily fixed $p\in \mathrm{Dom}(f)$  there exists a continuous map $r_p:N\rightarrow Z$ such that
\begin{enumerate}
\item $s(p)=r_p(f(p))$, 
\item $r_p\circ f:\mathrm{Dom}(f)\rightarrow Z$ is homotopic to $s|_{\mathrm{Dom}(f)}:\mathrm{Dom}(f)\rightarrow Z$ through a homotopy fixing $p$.
\end{enumerate}
\end{proposition}

\begin{proof}
Let $p\in  {\rm Dom} (f)\subset M$ be an arbitrarily fixed point. By Corollary \ref{cor:iso-bimero} we know that $f_*: \pi_1 (M,p)\to \pi_1 (N,f(p))$ is an isomorphism. Let
us consider the morphism $\pi_1 (N,f(p))\to \pi_1 (Z,s(p))$ equal to $s_*\circ f^{-1}_*$. By 
\cite[Prop. 1B.9, pg 90]{Allen} we know that there exists a continuous map $r_p:N\to Z$
sending $f(p)$ into $s(p)$ and unique up to homotopies fixing $f(p)$, such that
$$ (r_p)_*\circ f_*=s_*$$ as morphisms from $ \pi_1 (M,p)$ to $ \pi_1 (Z,s(p))$. By construction we have $r_p(f(p))=s(p)$ and the morphism $(r_p)_*\circ f_*:\pi_1(M,p)\rightarrow \pi_1(Z,s(p))$ equals $s_*:\pi_1(M,p)\rightarrow\pi_1(Z,s(p))$. Thus \cite{Allen} Prop. 1B.9 tells us that $r_p\circ f:\mathrm{Dom}(f)\rightarrow Z$ is homotopic to $s|_{\mathrm{Dom}(f)}:\mathrm{Dom}(f)\rightarrow Z$ through a homotopy fixing $p$.
\end{proof}

\begin{theorem}\label{2theo:main}
Let $s:M\to Z$ be a continuous map.
 For any continuous map $r:N\rightarrow Z$ satisfying the two properties of Def. \ref{main-definition}. we have 
 \begin{equation}\label{2main}
 s_* [\overline{\eth}_M]= r_*  [\overline{\eth}_N] \;\;\text{in}\;\;K^{\mathrm{an}}_0 (Z)
 \end{equation}
 where we recall that 
 $$K^{\mathrm{an}}_0 (Z)= \operatorname{dirlim}_{X \subset Z, X {\rm compact} } K^{\mathrm{an}}_0 (X)$$
 \end{theorem}

\begin{proof}
Let $r:N\rightarrow Z$ any continuous map satisfying the assumptions of Prop. \ref{homotopy}. By the very definition of bimeromorphic map $f:M\dasharrow N$ we know that there exists a compact irreducible analytic subvariety $\mathcal{G} (f) $ of $M\times N$ and a pair of modifications  $$p_M: \mathcal{G} (f) \to M\quad\text{and}\quad p_N: \mathcal{G} (f) \to N$$
induced by the natural projections of $M\times N$ onto the first and second factor respectively. Let $b: B\to \mathcal{G} (f)$ be a resolution of $ \mathcal{G} (f)$. By composing $b$ with $p_M$ and $p_N$ we  obtain a pair of modifications  $$M\stackrel{\beta_M}{\longleftarrow}B\stackrel{\beta_N}{\longrightarrow} N.$$ Let $q$ be a point in $B$ such that $b(q)=p$, where $p\in \mathrm{Dom}(f)$ and $r(f(p))=s(p)$. By the assumptions on $r$ we know that $r_*\circ f_*:\pi_1(M,p)\rightarrow \pi_1(Z,s(p))$ coincides with $s_*:\pi_1(M,p)\rightarrow \pi_1(Z,s(p))$. On the other hand, by definition, $f:\mathrm{Dom}(f)\rightarrow N$ equals $\beta_N\circ ({\beta_M}|_{\mathrm{Dom}(f)})^{-1}:\mathrm{Dom}(f)\rightarrow N$. Thus we can conclude that the following two morphisms of groups coincide $$(s\circ \beta_M)_* :\pi_1(B,q)\rightarrow \pi_1(Z,s(p))\ \quad\ \text{and}\ \quad\ (r\circ \beta_N)_*:\pi_1(B,q)\rightarrow \pi_1(Z,s(p)).$$ 
Indeed $$r_*\circ f_* = s_* \Leftrightarrow r_* \circ (\beta_N)_* \circ (\beta_M)_*^{-1}=s_* \Leftrightarrow (r\circ \beta_N)_*=
(s\circ \beta_M)_*\,.$$
Prop. 1B.9 in \cite{Allen} allows us to conclude that $s\circ \beta_M:B\rightarrow Z$ and $r\circ \beta_N: B\rightarrow Z$ are homotopic through a homotopy fixing $q$. Consequently, using \eqref{equality-bis} and the homotopy invariance of K-homology,
we have the following equalities in  $K_0(Z)$: $$r_*([\overline{\eth}_N])=r_*({\beta_N}_*([\overline{\eth}_B])=s_*({\beta_M}_*([\overline{\eth}_B]))=s_*([\overline{\eth}_M]).$$
\end{proof}

 \begin{corollary}\label{hgt}
 The higher Todd genera 
 $$\{ ({\rm Td}(M)\cup s^* \alpha)[M]\,,\,\alpha\in H^* (B\Gamma,\mathbb{Q})\}$$
 are bimeromorphic invariants of $M$ in the sense of  definition \ref{main-definition}.
 \end{corollary}
 
 \begin{proof}
Thanks to Th.\ref{2theo:main} we know that $s_* [\overline{\eth}_M]\in K^{{\rm an}}_* (B\Gamma)$ is a bimeromorphic invariant, in that  $s_* [\overline{\eth}_M]= r_*  [\overline{\eth}_N] \;\;\text{in}\;\;K^{\text{an}}_0 (B\Gamma)$. Thus 
 $\Ch_* (s_* [\overline{\eth}_M])\in H_* (B\Gamma,\mathbb{Q})$ is a bimeromorphic invariant.
 But
 $\Ch_* (s_* [\overline{\eth}_M])=s_* (\Ch_* [\overline{\eth}_M])$ and 
 \begin{equation}\label{ch=}
 \Ch_* [\overline{\eth}_M]= {\rm PD} ({\rm Td}(M))\,,
 \end{equation}
 with ${\rm PD} $ denoting Poincar\'e duality,
 from which the Corollary follows. Here we recall that the homology Chern character is defined as the transpose of the
 inverse of the usual Chern character in K-theory, denoted $\Ch^*$ (see Atiyah and Hirzebruch \cite{MAFH}),
 once we rationalize the K-theory groups and the singular (co)homology groups. The naturality of the homology Chern character follows from its definition and the naturality  of $\Ch^*$. The formula $\Ch_* [\overline{\eth}_M]= {\rm PD} ({\rm Td}(M))$ follows from  \cite[Lemma 11.4.1]{Higson-Roe:KHom} and the Atiyah-Singer index formula for twisted 
 Dirac operators.
  \end{proof}

  \noindent
  For interesting examples, see \cite[Section 5]{rosenberg-tams}.

\begin{remark}
We observe that, thanks to \eqref{ch=}, the higher Todd genera associated to $M$ and $s:M\to B\Gamma$
can also be written as follows:
 $$\{\langle \alpha, s_* (\Ch_* [\overline{\eth}_M])) \rangle\,,\quad \alpha\in H^* (B\Gamma,\mathbb{Q})\}.$$
\end{remark}

 \section{The singular case}
In this last section we extend the results of the previous section to certain singular varieties. We start with  few words on  Hermitian complex spaces. A paracompact  and reduced complex space $X$ is said to be \emph{Hermitian} if  the regular part of $X$ carries a Hermitian metric $h$  such that for every point $p\in X$ there exists an open neighborhood $U\ni p$ in $X$, a proper holomorphic embedding of $U$ into a polydisc $\phi: U \rightarrow \mathbb{D}^N\subset \mathbb{C}^N$ and a Hermitian metric $g$ on $\mathbb{D}^N$ such that $(\phi|_{\reg(U)})^*g=h$, see e.g. \cite{Takeo} or \cite{JRupp}. In this case we will write $(X,h)$ and with a little abuse of language we will say that $h$ is a \emph{Hermitian metric on $X$}. Clearly  analytic sub-varieties of  complex Hermitian manifolds endowed with the metric induced by the Hermitian  metric of the ambient space provide natural examples of  Hermitian complex spaces. Given a compact and irreducible Hermitian complex space $(X,h)$ let us introduce  the following  unbounded, densely defined and self-adjoint operator 
\begin{equation}
\label{Diracrel}
\overline{\eth}_{\mathrm{rel}}:L^2\Omega^{0,\bullet}(\reg(X),h)\rightarrow L^2\Omega^{0,\bullet}(\reg(X),h)
\end{equation}
  defined as $\overline{\eth}_{\mathrm{rel}}:=\overline{\partial}_{0,\min}+\overline{\partial}^t_{0,\max}$ where the latter operator is the rolled-up operator associated to the minimal $L^2$-$\overline{\partial}$-complex $(L^2\Omega^{0,q}(\reg(X),h),\overline{\partial}_{0,q,\min})$. We recall that the domain of \eqref{Diracrel} is $$\mathcal{D}(\overline{\eth}_{\mathrm{rel}})=\bigoplus_q\mathcal{D}(\overline{\partial}_{0,q,\min})\cap \mathcal{D}(\overline{\partial}^t_{0,q-1,\max})$$ $\overline{\partial}_{0,q,\min}:L^2\Omega^{0,q}(\reg(X),h)\rightarrow L^2\Omega^{0,q+1}(\reg(X),h)$ is defined as the graph closure of  $\overline{\partial}_{0,q,}:\Omega^{0,q}_c(\reg(X))\rightarrow \Omega_c^{0,q+1}(\reg(X))$ and $\overline{\partial}_{0,q,\max}^t:$ $L^2\Omega^{0,q+1}(\reg(X),h)\rightarrow $ $L^2\Omega^{0,q}(\reg(X),h)$ is the adjoint of $$\overline{\partial}_{0,q,\min}:L^2\Omega^{0,q}(\reg(X),h)\rightarrow L^2\Omega^{0,q+1}(\reg(X),h).$$ According to \cite[Cor. 5.2]{FraBei} or \cite[Th. 1.2]{OvrRup}  we know that \eqref{Diracrel}  has entirely discrete spectrum if $\dim(\sing(X))=0$ and thus it defines a class $[\overline{\eth}_{\mathrm{rel}}]\in K_0^{\mathrm{an}}(X):=KK_0(C(X),\mathbb{C})$, see \cite[Prop. 3.6]{BeiPiazza2019}.

\begin{definition}\label{definition-htgsingular}
Let $V$ be a compact and irreducible  Hermitian complex space with only  isolated singularities. Let $\Gamma$ be any discrete finitely generated group and let $s:V\to B\Gamma$ be a continuous map. The numbers 
$$\{\langle\alpha, s_* (\Ch_* [\overline{\eth}^V_{\mathrm{rel}}])\rangle\,,\;\alpha\in H^* (B\Gamma,\mathbb{Q})\}$$ 
are the relative higher analytic Todd genera of $M$ and $s:M\to B\Gamma$.
\end{definition}

We point out that if $V$ has only rational singularities and  $\dim(\sing(V))=0$ then $\alpha([\mathcal{O}_X])=[\overline{\eth}_{\mathrm{rel}}]$, that is, the Levy class equals the relative class, see \cite[Prop. 6.2]{BeiPiazza2019}\footnote{The assumption that $X$ has only rational singularities has to be added to the statement of \cite[Prop. 6.2]{BeiPiazza2019}}. Hence in this case we can replace $[\overline{\eth}_{\mathrm{rel}}]$ with $\alpha([\mathcal{O}_X])$ in Def. \ref{definition-htgsingular}. We remind the reader the $V$ has only rational singularities if it is normal and  for some resolution $\pi:M\rightarrow V$ (and hence for all) we have  $R^k\pi_*\mathcal{O}_M=0$, $k>0$. Log-terminal singularities, canonical singularities and toric singularities provide examples of rational singularities.
Consider now a pair of compact and irreducible Hermitian complex spaces $(V,h)$ and $(W,k)$ both with $\dim(\sing(V))=\dim(\sing(W))=0$. Assume that there exists  a  bimeromorphic map $\psi: V\dashrightarrow W$  and  that both $V$ and $W$ admit resolutions that preserve the fundamental group. Put it differently, there exist resolutions  $\pi:M\rightarrow V$ and $\rho:N\rightarrow W$ such that both maps $\pi_*:\pi_1(M)\rightarrow \pi_1(V)$ and $\rho_*:\pi_1(N)\rightarrow \pi_1(W)$ are isomorphisms. Note that if the fundamental group is preserved by a resolution then it is preserved by any resolution. Examples are provided by projective varieties with log-terminal singularities, see for instance \cite{Takayama} and the references therein.
 Let $\phi:M\dasharrow N$ be the bimeromorphic map induced by $\pi$, $\psi$ and $\rho$. Note that the composition $\rho_*\circ \phi_*\circ (\pi_*)^{-1}:\pi_1(V)\rightarrow \pi_1(W)$ is an isomorphism. Let  $Z$ be any $K(\Gamma,1)$ space, $\Gamma$ any discrete finitely generated group, and let $s:V\rightarrow Z$ be any continuous map.

 \begin{definition}\label{main-definition-bis}
Let $\psi:V\dashrightarrow W$ be as above, in particular both $V$ and $W$ admit resolutions preserving the fundamental group.  Let $s:V\to B\Gamma$  be any continuous  map and let $r:W\rightarrow B\Gamma$ be a continuous map such that for a pair of resolutions $\pi:M\rightarrow V$ and $\rho:N\rightarrow W$ it holds:
\begin{enumerate}
\item $s(\pi(p))=r (\rho(\phi(p)))$ for some $p\in \mathrm{Dom}(\phi)$, with $\phi:M\dasharrow N$ 
 the bimeromorphic map induced by $\pi$, $\psi$ and $\rho$,
\item $r \circ \rho\circ \phi:\mathrm{Dom}(\phi)\rightarrow Z$ is homotopic to $s\circ \pi|_{\mathrm{Dom}(\phi)}:\mathrm{Dom}(\phi)\rightarrow Z$ through a homotopy fixing $p$.
\end{enumerate}
By bimeromorphic invariance of the  relative higher analytic Todd genera
$$\{\langle\alpha, s_* (\Ch_* [\overline{\eth}^V_{\mathrm{rel}}])\rangle\,,\;\alpha\in H^* (B\Gamma,\mathbb{Q})\}$$
we mean the equality
\begin{equation}
\label{2wemean-bis}
\langle\alpha, s_* (\Ch_* [\overline{\eth}^V_{\mathrm{rel}}])\rangle=
\langle \alpha, r_* (\Ch_* [\overline{\eth}^W_{\mathrm{rel}}])\rangle
\end{equation}
for each $[\alpha]\in   H^*(B\Gamma,\mathbb{Q})$ and for each  $r:W\rightarrow B\Gamma$ satisfying $(1)$ and $(2)$ above.
\end{definition}

 Proceeding exactly as in  the proof of  Prop. \ref{homotopy} we have the following property:
 for any arbitrarily fixed $p\in \mathrm{Dom}(\phi)$  there exists a continuous map $r_p:W\rightarrow Z$ satisfying (1) and (2) 
 in Definition \ref{main-definition-bis}.
 In other words $r_p\circ \rho:N\rightarrow Z$ and $s\circ \pi:M\rightarrow Z$ satisfy Prop. \ref{homotopy} with respect to $\phi:M\dashrightarrow N$ and $p\in \mathrm{Dom}(\phi)$.

\bigskip

 We have now all the ingredients for the following 
 \begin{proposition}\label{2sTability}
Let $\psi:V\dashrightarrow W$ be as above and let $s:V\to B\Gamma$  be a continuous  map. For any continuous map $r:W\rightarrow Z$ satisfying the two properties listed in Def. \ref{main-definition-bis} we have 
 \begin{equation}\label{2main}
 s_* [\overline{\eth}_{\mathrm{rel}}^V]= r_*  [\overline{\eth}_{\mathrm{rel}}^W] \;\;\text{in}\;\;K^{\mathrm{an}}_0 (Z)
 \end{equation}
\end{proposition}

\begin{proof}
Thanks to   \cite[Th. 4.1]{BeiPiazza2019} we know that $[\overline{\eth}_{\mathrm{rel}}]=\pi_*[\overline{\eth}_M]$, where $\pi:M\rightarrow X$ is any resolution of $X$. Now using \cite[Th. 4.1]{BeiPiazza2019} and  Th. \ref{2theo:main} we have $$ s_* [\overline{\eth}_{\mathrm{rel}}^V]=s_*(\pi_* [\overline{\eth}_M])=r_*(\rho_*[\overline{\eth}_N])=r_*[\overline{\eth}_{\mathrm{rel}}^W].$$
\end{proof}

From Proposition \ref{2sTability} we obtain immediately the following Corollary

 \begin{corollary}\label{2sTability-bis}
 The  relative higher analytic Todd genera
$$\{\langle\alpha, s_* (\Ch_* [\overline{\eth}^V_{\mathrm{rel}}])\rangle\,,\;\alpha\in H^* (B\Gamma,\mathbb{Q})\}$$ 
are bimeromorphic invariants in the sense of Definition \ref{main-definition-bis}.
\end{corollary}
For more on analytic Todd genera in the singular setting we refer the reader to \cite{BeiPiazza2019}.

\begin{remark}
Proposition \ref{2sTability} extends to Hermitian complex spaces, but also reformulate, Proposition 7.1
 in  \cite{BeiPiazza2019}. We warn the reader that the proof of Proposition 7.1 in 
   \cite{BeiPiazza2019} is not correct \footnote{in  Lemma 7.2  in    \cite{BeiPiazza2019}
  we cannot conclude that $\ell=\pi\circ r$ up to homotopy, as $\tilde{M}$ and $\pi^*\tilde{V}$ are isomorphic as coverings but not as principal bundles.} and that  Proposition \ref{2sTability} 
  and Corollary \ref{2sTability-bis} are now the correct version of that result.
\end{remark}

\begin{acknowledgements}
This paper was partially written while the first author was a postdoc at the department of Mathematics of the University of Padova. He wishes to thank that institution  for financial support. We thank Jonathan Rosenberg and J\"org Sch\"urmann for useful comments and interesting discussions.
\end{acknowledgements}


\begin{thebibliography}{9}

\bibitem{Factorization}
Abramovich, D.,  Karu, K.,  Matsuki, K., W\l{o}darczyk, J.: 
Torification and factorization of birational maps.  J. Amer. Math. Soc. \textbf{15}, 531--572 (2002)
 
\bibitem{MAFH}
Atiyah, M. F., Hirzebruch, F.:
Vector bundles and homogeneous spaces. In Proc. Sympos. Pure Math., Volume~III, American Mathematical Society, 7--38, Providence, R.I., (1961).
   
\bibitem{FraBei}
Bei, F.:
Degenerating Hermitian metrics and spectral geometry of the canonical bundle.
Adv. Math. \textbf{328}, 750--800, (2018)
      
\bibitem{BeiPiazza2019}
Bei, F., Piazza, P.: 
On analytic Todd classes of singular varieties.
To appear on International Mathematics Research Notices, https://doi.org/10.1093/imrn/rnz232
   

\bibitem{WBlock}
Block, J., Weinberger, S.:
Higher Todd classes and holomorphic group actions.
Pure Appl. Math. Q. \textbf{2}, 1237--1253, (2006).
     
     
\bibitem{BSYarxiv}
Brasselet, J. P.,  Sch\"{u}rmann, J., Yokura, S.:
Hirzebruch classes and motivic Chern classes for singular spaces. https://arxiv.org/abs/math/0503492


\bibitem{BraSchYou}
Brasselet, J. P.,  Sch\"{u}rmann, J., Yokura, S.:
Hirzebruch classes and motivic Chern classes for singular spaces.
J. Topol. Anal. \textbf{2}, 1--55, (2010)
     
\bibitem{Chirka}
Chirka, E. M.:
Complex analytic sets.
Mathematics and its Applications (Soviet Series), Volume~46, Translated from the Russian by R. A. M. Hoksbergen, Kluwer Academic Publishers Group, Dordrecht, (1989).
     
\bibitem{Allen}
Hatcher, A.:
Algebraic topology.
Cambridge University Press, Cambridge, (2002)


\bibitem{Higson-Roe:KHom}
Higson, N., Roe J.:
Analytic {$K$}-homology.
Oxford Mathematical Monographs, Oxford Science Publications, Oxford University Press, Oxford, (2000)
     
\bibitem{Hilsumproj}
Hilsum, M.:
Une preuve analytique de la conjecture de J. Rosenberg.
https://hal.archives-ouvertes.fr/hal-01841905v1

\bibitem{Roni}
Levy, R. N.:
The Riemann-Roch theorem for complex spaces.
Acta Math. \textbf{158}, 149--188, (1987)
     
\bibitem {Takeo}
Ohsawa, T.:
$L^2$ approaches in several complex variables, Springer Monographs in Mathematics, Development of Oka-Cartan theory by $L^2$ estimates for the $\overline{\partial}$ operator, Springer, Tokyo, (2015)
     
\bibitem{OvrRup}
{\O}vrelid, N., Ruppenthal, J.:
$L^2$-properties of the $\overline{\partial}$ and the $\overline{\partial}$-Neumann operator on spaces with isolated singularities.
Math. Ann. \textbf{359}, 803--838, (2014)
    
      
\bibitem{Peter}
Peternell, T.:
Modifications. In Several complex variables, VII, Encyclopaedia Math. Sci., Volume~74, 285--317, Springer, Berlin, (1994)
   
\bibitem{rosenberg-tams}
Rosenberg, J.:
An analogue of the {N}ovikov conjecture in complex algebraic geometry.
Trans. Amer. Math. Soc. \textbf{360}, 383--394, (2008)
    
\bibitem{JRupp}
Ruppenthal, J.:
$L^2$-theory for the $\overline\partial$-operator on compact complex spaces.
Duke Math. J. \textbf{163}, 2887--2934, (2014)
     
\bibitem{Takayama}
Takayama, S.:
Local simple connectedness of resolutions of log-terminal singularities.
Internat. J. Math. \textbf{14}, 825--836, (2003)
  
\bibitem{Tomassini}
Tomassini, G.:
Structure theorems for modifications of complex spaces.
Rend. Sem. Mat. Univ. Padova. \textbf{59}, 295--306, (1978)
      
     
\bibitem{Ueno}
Ueno, K.:
Classification theory of algebraic varieties and compact complex spaces. Lecture Notes in Mathematics \textbf{439}, Springer-Verlag, Berlin-New York (1975)




\end{thebibliography}
\end{document}